\documentclass[absolute]{mymathart}
\usepackage[theorems]{mymathmacros}
\usepackage[usenames,dvipsnames]{color}
\usepackage{subfig}
\usepackage[shortlabels]{enumitem}
\usepackage{hyperref}
\usepackage{amssymb}
\usepackage{graphicx}
\usepackage{epigraph}
\usepackage{amsfonts,amssymb,color}
\usepackage[mathscr]{eucal}
\usepackage{amsmath, amsthm, amscd}
\usepackage{mathrsfs}
\usepackage{amsbsy}
\usepackage{float}
\usepackage{verbatim}
\usepackage{amsfonts} 
\usepackage{pifont}

\title[Topological description of the Borel probability space]{Topological description of the Borel probability space}

\author{Liangang Ma}

\address{Dept.\ of Mathematical Sciences, Binzhou University, Huanghe 5th Road No. 391, Binzhou 256600, Shandong, P. R. China} 
\email{maliangang000@163.com}
\thanks{The work is supported by ZR2019QA003 from SPNSF and 12001056 from NSFC}    

\newtheorem{theorem}[subsection]{Theorem}
\newtheorem{lemma}[subsection]{Lemma}
\newtheorem{Young's Lemma}[subsection]{Young's Lemma}
\newtheorem{Urysohn's Lemma}[subsection]{Urysohn's Lemma}
\newtheorem{SSU Theorem}[subsection]{Simon-Solomyak-Urba\'nski Theorem}
\newtheorem{Prohorov's Theorem}[subsection]{Prohorov's Theorem}
\newtheorem{proposition}[subsection]{Proposition}
\newtheorem{pro}[subsection]{Problem}
\newtheorem{corollary}[subsection]{Corollary}
\newtheorem{definition}[subsection]{Definition}
\newtheorem{exm}[subsection]{Example}
\newtheorem{rem}[subsection]{Remark}

\numberwithin{equation}{section}

\begin{document} 

\begin{abstract}
   
We study properties of some popular topology on the space of Borel probabilities on a topological ambient space in this paper.  We show that the two types of popular vague topology are equivalent to each other in case the ambient space is LCH. The two types of setwise topology induced from two equivalent descriptions of setwisely sequential convergence of probability measures are also equivalent to each other regardless of the topology on the ambient space. We give explicit conditions for the two types of vague topology and the two types of setwise topology to be separable or metrizable on the space of Borel probabilities. These conditions are either in terms of the cardinality of the elementary events in the Borel $\sigma$-algebra or some direct topological assumptions on the ambient space. We give an necessary and sufficient condition for families of probability measures to be setwisely relatively compact in case the ambient space is a compact metric space. There are some extending problems and heuristic schemes on formulating new topologies on the space of Borel probabilities at the end of the work.   

\end{abstract}
 
 \maketitle

\section{Introduction}

This work addresses the following question: Let $X$ be a topological space along with its Borel $\sigma$-algebra $\mathcal{B}$, consider the collection of all the probability measures $\mathcal{M}(X)$ on $(X,\mathcal{B})$, how to describe the topological structure of  $\mathcal{M}(X)$? 

The subtlety of the topology on $(X,\mathcal{B})$ is revealed when one considers problems related to regularity (for example, continuity) of mapping on the probability space $\mathcal{M}(X)$, which is even crucial in attacking these problems in some cases. For long time importance of the weak topology is recognized in its various applications. For example, see its application in interactive particle systems by J. T. Cox, A. Klenke and E. A. Perkins on a locally compact Polish space in \cite{CK, CKP}. The solution of problems due to applications of various other topology in due course indicates the importance of these topology on $\mathcal{M}(X)$. While various other topologies on $\mathcal{M}(X)$ are highlighted in their applications, we limit our attention to the vague, weak, setwise and TV topology in this work.   

The vague topology on $\mathcal{M}(X)$ is the coarsest topology among the four. We focus on two types of vague topology on $\mathcal{M}(X)$. One is in the sense of Kallenberg \cite{Kallen1, Kallen4}, another is in the sense of Folland \cite{Fol}. We will show the equivalence of the two types of vague topology in case of the ambient space $X$ is locally compact Hausdorff (LCH). We also study some topological properties of $\mathcal{M}(X)$ under the vague topology, such as its separability and metrizability. 

Setwise topology appears as a finer topology than the weak topology on $\mathcal{M}(X)$. The sequential convergence under this topology is a more demanding property than under the weak topology. See its application in the Markov decision processes by E. Feinberg, P. Kasyanov and M. Zgurovsky \cite{FK, FKL, FKZ3, FKZ4}, in the Markov chains by O. Hern\'andez-Lerma and J. Lasserre \cite{HL1, HL3}. The notion is further exploited by the author into the dimensional theory of (families of) iterated function systems (IFS), which extends some important results in this field \cite{Ma1, Ma2}. 


We derive two types of setwise topology, which follow naturally from two equivalent descriptions of sequential setwise convergence of probability measures on $\mathcal{M}(X)$. The two types of setwise topology are always equivalent to each other. The separability and metrizability of $\mathcal{M}(X)$ under the setwise topology are decided by the cardinality of elementary events in the Borel $\sigma$-algebra $\mathcal{B}$. We consider the setwisely relative compactness of families of probabilities in $\mathcal{M}(X)$ following Billingsley and Kallenberg. 

There are three notes  for the readers. The first one is on the ambient space $X$. We try to set our results on general topological ambient spaces, while some results do require some assumptions on the ambient spaces. Notable assumptions are those requiring $X$ to be LCH, compact, metrizable, separable, normal or complete.

The second note is that although we pay main attention to the vugue and setwise topology on $\mathcal{M}(X)$ (partly because the weak topology is well-studied in various circumstances), we hope it will shed some light on properties of other topology, as well as comparisons of different types of topology induced from equivalent descriptions of convergence of sequences of probability measures on $\mathcal{M}(X)$.

The last one is on the scope of measures that our results apply. Let $\hat{\mathcal{M}}(X)$  be the collection of all finite measures on $(X,\mathcal{B})$. Some notions and results extend naturally from $\mathcal{M}(X)$ to $\hat{\mathcal{M}}(X)$ by normalization, or even to the space of all measures (including infinite ones) in some cases.

The organization of the paper is as following. In Section 2 we introduce the vague, weak, setwise and TV topology on $\mathcal{M}(X)$. These definitions are induced naturally from the corresponding sequential convergence of probability measures in $\mathcal{M}(X)$. The main results of the work are presented in this section-Theorem \ref{thm4}, \ref{thm3}, \ref{thm2}, \ref{thm7}, \ref{thm8}, \ref{thm1}.  Section 3 is devoted to the proof of Theorem \ref{thm4} and \ref{thm7}, which compare different types of vague and setwise topology on $\mathcal{M}(X)$. In Section 4 we prove Theorem \ref{thm3}, \ref{thm2} and \ref{thm8}, which give conditions on the separability and metrizability of $\mathcal{M}(X)$ under the vague or setwise topology in due course. Section 5 is devoted to the proof of Theorem \ref{thm1} on the setwisely relative compactness of families of probabilities in $\mathcal{M}(X)$. In section 6 we indicate some further problems on the notable topology on $\mathcal{M}(X)$. In the last section we propose some open schemes to defining more subtle topologies in ones' needs.  


\section{The vague, weak, setwise, TV topology on $\mathcal{M}(X)$ and the main results}

We introduce concepts of the vague, weak, setwise and total-variation (TV) topology on $\mathcal{M}(X)$ in this section. These notions differ from each other in fineness. One can expect more desiring properties for convergent sequences of probability measures under finer topology. For example, the \emph{uniform Fatou lemma} holds for convergent sequences of finite measures under the TV topology, but it does not hold for convergent sequences of finite measures under the setwise topology on $\mathcal{M}(X)$ \cite{FKZ2}.

We merely assume the ambient space $X$ is a topological space with its Borel $\sigma$-algebra $\mathcal{B}$, in general. Some assumptions on the ambient space $X$ are required in order to deduce some particular results. For integrations on a general measure space, see \cite{Li} or \cite[Chapter 3]{Tay}.  A function $f: X\rightarrow\mathbb{R}$ is said to \emph{vanish at infinity} if 
\begin{center}
$f^{-1}\big((-\infty,-\epsilon]\cup[\epsilon,\infty)\big)$ 
\end{center}
is a compact set in $X$ for any $\epsilon>0$. The \emph{support} of a function $f: X\rightarrow\mathbb{R}$ is defined to be the closure
\begin{center}
$Supp(f):=\overline{f^{-1}\big((-\infty,0)\cup(0,\infty)\big)}$.  
\end{center} 
We highlight the following families of testing functions in this work.

\begin{itemize}
\item $C(X)=\{f: f \mbox{ is a continuous function from } X \mbox{ to } \mathbb{R}\}.$
\item $C_b(X)=\{f: f \mbox{ is a bounded continuous function from } X \mbox{ to } \mathbb{R}\}$.
\item $C_0(X)=\{f: f \mbox{ is a continuous function from } X \mbox{ to } \mathbb{R} \mbox{ vanishing  at infinity} \}$.
\item $C_{sc}(X)=\{f: f \mbox{ is a continuous function from } X \mbox{ to } \mathbb{R} \mbox{ with compact support}\}$.
\item $C_{sb}(X)=\{f: f \in C(X) \mbox{ has bounded support in case } X \mbox{ is metrizable}\}$.
\item $M_b(X)=\{f: f \mbox{ is a bounded measurable function from } X \mbox{ to } \mathbb{R}\}$.
\item $M_1(X)=\{f: f \mbox{ is a measurable function from } X \mbox{ to } [-1,1]\}$.
\end{itemize}
Obviously we can see that,
\begin{equation}\label{eq1}
M_1(X)\cup C_b(X)\subset M_b(X) \mbox{ and } C_{sc}(X)\cup C_{sb}(X)\subset C_0(X)\subset C_b(X)\subset C(X).
\end{equation}
If the ambient space $X$ is a metric space satisfying the \emph{Heine-Borel Property} \cite{JW}, then 
\begin{center}
$C_{sc}(X)=C_{sb}(X)$.
\end{center}
Refer to \cite{Kallen1} for  testing functions in the family $C_{sb}(X)$.

We pay attention to the following two types of vague topology on $\mathcal{M}(X)$ in this work, following  two popular notions of vaguely sequential convergence of probability measures in $\mathcal{M}(X)$.
\begin{definition}
The \emph{Type-I vague topology} $\mathfrak{W}_{v1}$ on $\mathcal{M}(X)$ is the topology with basis 
\begin{center}
$W_{v1}(\nu,f, \epsilon)=\{\varrho\in\mathcal{M}(X): |\int_X f(x) d\varrho-\int_X f(x) d\nu|<\epsilon \}$
\end{center}
for any $f\in C_{sc}(X)$ and any real $\epsilon>0$. 
\end{definition}

This type of vague topology is induced from the vaguely sequential convergence of probability measures in \cite[Chapter 4]{Kallen1} and \cite[Definition 13.12]{Kle}.

\begin{definition}
The \emph{Type-II vague topology} $\mathfrak{W}_{v2}$ on $\mathcal{M}(X)$ is the topology with basis 
\begin{center}
$W_{v2}(\nu,f, \varepsilon)=\{\varrho\in\mathcal{M}(X): |\int_X f(x) d\varrho-\int_X f(x) d\nu|<\epsilon \}$
\end{center}
for any $f\in C_0(X)$ and any real $\epsilon>0$. 
\end{definition}

This type of vague topology is induced from the vaguely sequential convergence of probability measures in \cite{Fol} and \cite{Las1}. It is obvious that the Type-II vague topology is finer than Type-I vague topology since $C_{sc}(X)\subset C_0(X)$. Our first result shows that the fineness is not strict in some cases.
\begin{theorem}\label{thm4}
If the ambient space $X$ is an LCH space, then the  Type-II vague topology $\mathfrak{W}_{v2}$ is equivalent to the Type-I vague topology $\mathfrak{W}_{v1}$ on $\mathcal{M}(X)$.
\end{theorem}

Now rename the space $\mathcal{M}(X)$ as $\mathcal{M}_{v1}(X)$ and $\mathcal{M}_{v2}(X)$  under the topology $\mathfrak{W}_{v1}$ and $\mathfrak{W}_{v2}$ respectively. Our next result describes the separability and metrizability of the probability space $\mathcal{M}(X)$ under the vague topology. A set $A\in \mathcal{B}$ is called an \emph{elementary event} if it does not contain any non-empty proper subset $B\in\mathcal{B}$.

\begin{theorem}\label{thm3}
For an LCH space $X$, $\mathcal{M}_{v1}(X)$ ($\mathcal{M}_{v2}(X)$) is separable and metrizable if the Borel $\sigma$-algebra $\mathcal{B}$ admits at most countably many elementary events.  
\end{theorem}

The conclusion is not true for non-LCH space $X$, see Remark \ref{rem1}. It is natural to ask the separability and metrizablility of $\mathcal{M}_{v1}(X)$ or $\mathcal{M}_{v2}(X)$ in case the Borel $\sigma$-algebra $\mathcal{B}$ has uncountably many elementary events. We are not able to provide a conclusive answer to this question, even if $X$ is LCH. However, we do have some results when $X$ is a compact metric space. The separability of the Borel probability space under the vague topology can be found in \cite[Exercise 1.10.22.]{Tao}, see also Proposition \ref{pro1}.

\begin{theorem}\label{thm2}
The probability space $\mathcal{M}_{v1}(X)$ ($\mathcal{M}_{v2}(X)$) is separable and metrizable if the ambient space $X$ is a compact metric space.  
\end{theorem}

\begin{rem}
The readers are strongly recommended to \cite[Theorem 4.2]{Kallen1} for a stronger result on the separability, metrizability and completeness of the space $\mathcal{M}_{v1}(X)$ ($\mathcal{M}_{v2}(X)$), with $X$ being a separable and complete metric space (a compact metric space is separable and complete). The author finds Kallenberg's exciting result after Theorem \ref{thm2}, but I will retain it in its form as the proof differs from Kallenberg's techniques completely, which will benefit the readers from different point of views. See also Problem \ref{prob2}. 
\end{rem}

Now we turn to the weak topology on $\mathcal{M}(X)$. 

\begin{definition}
The \emph{Type-I weak topology} $\mathfrak{W}_{w1}$ on $\mathcal{M}(X)$ is the topology with basis 
\begin{center}
$W_{w1}(\nu,f, \epsilon)=\{\varrho\in\mathcal{M}(X): |\int_X f(x) d\varrho-\int_X f(x) d\nu|<\epsilon \}$
\end{center}
for any $f\in C_b(X)$ and any real $\epsilon>0$. 
\end{definition}

This topology is obviously finer than the two types of vague topology $\mathfrak{W}_{v1}$ and $\mathfrak{W}_{v2}$ on $\mathcal{M}(X)$. There is a detailed study of some finer form of weak topology and its various applications in  \cite{Kal}. The topology $\mathfrak{W}_{w1}$ is metrizable if $X$ is metrizable, for example, by the Prohorov metric (see \cite[Section 6]{Bil1}). The Type-I weak topology is well-understood, so will not be our focus in this work. See Problem \ref{prob1} on comparing it with other types of weak topology on $\mathcal{M}(X)$.

In the following we introduce two types of setwise topology on $\mathcal{M}(X)$. Different from the introduction of the vague and weak topology above, we would like to do this from the setwisely sequential convergence of probability measures in $\mathcal{M}(X)$. 

\begin{definition}\label{def1}
A sequence of probability measures $\{\nu_n\in\mathcal{M}(X)\}_{n=1}^\infty$ is said to converge \emph{setwisely} to $\nu\in\mathcal{M}(X)$, if 
\begin{center}
$\lim_{n\rightarrow\infty }\nu_n(A)=\nu(A)$
\end{center}
for any $A\in\mathcal{B}$. 
\end{definition}

See for example \cite{Doo, GR, HL1, Las1}. Denote the sequential convergence in this sense by 
\begin{center}
$\nu_n\stackrel{s}{\rightarrow}\nu$ 
\end{center}
as $n\rightarrow\infty$. 

\begin{definition}\label{def3}
The induced topology $\mathfrak{W}_{s1}$ with subbasis
\begin{center}
$W_{s1}(\nu, A, \epsilon)=\{\varrho\in\mathcal{M}(X): |\varrho(A)-\nu(A)|<\epsilon \}$
\end{center}
for $A\in\mathcal{B}$ and any real $\epsilon>0$ is called the Type-I \emph{setwise topology} on $\mathcal{M}(X)$. 
\end{definition}

An equivalent way on describing the setwisely sequential convergence of measures stems from treating the measure as a  functional on the space of bounded Borel-measurable functions on $X$. 

\begin{definition}\label{def2}
A sequence of probability measures $\{\nu_n\in\mathcal{M}(X)\}_{n=1}^\infty$ is said to converge \emph{setwisely} to $\nu\in\mathcal{M}(X)$, if 
\begin{center}
$\lim_{n\rightarrow\infty }\int_X f(x) d\nu_n=\int_X f(x) d\nu$
\end{center}
for any $f\in M_b(X)$. 
\end{definition}
This is because the simple functions are dense among the bounded Borel-measurable functions on $X$ under the supremum norm. In this way one can define a topology $\mathfrak{W}_{s2}$ on $\mathcal{M}(X)$ as following.

\begin{definition}\label{def4}
The Type-II \emph{setwise topology} $\mathfrak{W}_{s2}$ on $\mathcal{M}(X)$ is the topology with basis 
\begin{center}
$W_{s2}(\nu,f, \epsilon)=\{\varrho\in\mathcal{M}(X): |\int_X f(x) d\varrho-\int_X f(x) d\nu|<\epsilon \}$
\end{center}
for any $f\in M_b(X)$ and any real $\epsilon>0$. 
\end{definition}

Feinberg, Kasyanov and Zgurovsky gave some equivalent conditions on verifying the setwisely sequential convergence of probability measures in $\mathcal{M}(X)$ with $X$ being metrizable as following (refer to \cite[Theorem 2.3]{FKZ1}).

\begin{theorem}[Feinberg-Kasyanov-Zgurovsky]\label{thm9}
For a sequence of measures $\{\nu_n\in\mathcal{M}(X)\}_{n=1}^\infty$ and $\nu\in\mathcal{M}(X)$ with a metric space $X$, the following conditions are equivalent to each other:
\begin{enumerate}[(I).]
\item $\nu_n\stackrel{s}{\rightarrow}\nu$ as $n\rightarrow\infty$.
\item $\lim_{n\rightarrow\infty} \nu_n(A)= \nu(A)$ for any open set $A\subset X$.
\item $\lim_{n\rightarrow\infty} \nu_n(A)= \nu(A)$ for any closed set $A\subset X$.
\end{enumerate}
\end{theorem}  

There are some sufficient conditions to guarantee setwisely sequential convergence of measures in certain contexts by Feinberg-Kasyanov-Zgurovsky and J. Lasserre, see \cite[Section 3]{FKZ1} and \cite[Lemma 4.1(i)(ii)]{Las1}. See also the Vitali-Hahn-Saks Theorem (refer to \cite{Doo} or \cite{HL2}) on setwisely sequential convergence of measures in $\mathcal{M}(X)$.  Due to the equivalence of Definition \ref{def1} and \ref{def2}, it is alluring whether the two types of topology $\mathfrak{W}_{s1}$ and $\mathfrak{W}_{s2}$ are equivalent to each other on $\mathcal{M}(X)$. 
 
\begin{theorem}\label{thm7}
The Type-I setwise topology  $\mathfrak{W}_{s1}$ is always equivalent to the Type-II setwise topology  $\mathfrak{W}_{s2}$ on the probability space $\mathcal{M}(X)$. 
\end{theorem}

Note that two topologies on a same space admit the same convergent sequences does not guarantee they are equivalent to each other, since there are examples of spaces with inequivalent topology which admit the same convergent sequences (as one can refer to an example by J. Schur \cite{Sch}). Especially one needs to be careful about the case when the two types of setwise topology are not metrizable (refer to Theorem \ref{thm8}).

In the following we pay attention to the separability and metrizability of the probability space $\mathcal{M}(X)$ under the setwise topology. Rename the space $\mathcal{M}(X)$ as  $\mathcal{M}_{s1}(X)$ equipped with the topology $\mathfrak{W}_{s1}$, and $\mathcal{M}_{s2}(X)$ equipped with the topology $\mathfrak{W}_{s2}$. Part of the following result is due to J. K. Ghosh and R. V. Ramamoorthi \cite[Proposition 2.2.1]{GR}.
\begin{theorem}\label{thm8}
The topological space $\mathcal{M}_{s1}(X)$ ($\mathcal{M}_{s2}(X)$) is separable or metrizable if and only if the Borel $\sigma$-algebra $\mathcal{B}$ admits at most countably many elementary events.  
\end{theorem}

Now we consider setwisely relative compactness of families of probability measures in  $\mathcal{M}(X)$. A family of probability measures $\Xi\subset \mathcal{M}(X)$ is said to be \emph{setwisely relatively compact} if there is a setwisely convergent subsequence in every sequence of probability measures in $\Xi$. We give a necessary and sufficient condition for families of probability measures to be setwisely relatively compact when the ambient space $X$ is a compact metric space.  

\begin{theorem}\label{thm1}
For a compact metric space $X$, a family of probability measures $\Xi\subset \mathcal{M}(X)$ is setwisely relatively compact if and only if for any sequence of probability measures $\{\nu_n\}_{n=1}^\infty\subset \Xi$, there is a subsequence $\{\nu_{n_i}\}_{i=1}^\infty$, such that
\begin{equation}\label{eq3}
\limsup_{i\rightarrow\infty} \nu_{n_i}(U)=\sup_{K\subset U, K \mbox{ is closed}}\limsup_{i\rightarrow\infty} \nu_{n_i}(K)
\end{equation}
for any open set $U\subset X$.
\end{theorem}

In case the ambient space $X$ is not a compact metric space, it seems more complicated to judge setwisely relative compactness of a family of probability measures in  $\mathcal{M}(X)$. See Problem \ref{prob3}.

The last topology we consider in this work is the total-variation (TV) topology  on  $\mathcal{M}(X)$. It is induced by the total-variation (TV) metric on $\mathcal{M}(X)$.

\begin{definition}
The \emph{total-variation metric}  is defined by
\begin{center}
$\Vert \nu-\varrho\Vert_{TV}=\sup_{A\in\mathcal{B}}\{|\nu(A)-\varrho(A)|\}$
\end{center}   
between two measures $\nu, \varrho\in\mathcal{M}(X)$.
\end{definition}

One is recommended to \cite{Doo, FKZ1, HL1, PS}) for more related topics on the TV topology on $\mathcal{M}(X)$. The Type-I (Type-II) setwise topology is coarser than the TV topology, and this comparison is strict in some cases. In fact, there are examples of sequences of Borel probability measures converging under the setwise topology but diverging under the $TV$ topology on compact metric ambient spaces.

\section{Comparison of the vague topology $\mathfrak{W}_{v1}$ with $\mathfrak{W}_{v2}$ and the setwise topology $\mathfrak{W}_{s1}$ with $\mathfrak{W}_{s2}$ on $\mathcal{M}(X)$}

This section is devoted to the proof of Theorem \ref{thm4} and   Theorem \ref{thm7} on comparison of different types of vague and setwise topology on $\mathcal{M}(X)$. To prove Theorem \ref{thm4}, we need some suitable continuous approximation for characteristic functions of compact sets in LCH spaces. The following result is a locally compact version of Urysohn's Lemma (refer to \cite[p131]{Fol}).   

\begin{Urysohn's Lemma}
Let $X$ be an LCH space. For any compact set $K\subset X$, let $U\supset K$ be its open neighbourhood, then there exists a continuous map $g_K: X\rightarrow [0,1]$ with compact support $Supp(g_K)\subset U$, such that
\begin{center}
$g_K(x)=1$
\end{center} 
for any $x\in K$.
\end{Urysohn's Lemma}

Now we recall the following separation axioms in general topological spaces.
\begin{itemize}

\item A topological space $X$ is called \emph{Hausdorff} or $T_2$ if for any two points $x,y\in X$, there are two disjoint open sets containing $x$ and  $y$ respectively.

\item A topological space $X$ is called \emph{regular} if for any point $x\in X$ and any closed set $C\subset X$ not containing $x$, there are two disjoint open sets containing $x$ and  $C$ respectively.

\item A topological space $X$ is called $T_3$ if it is both Hausdorff and regular.

\end{itemize}

Note that some people use the terminology 'regular' for 'regular and Hausdorff' ($T_3$) spaces.  A topological space $X$ is called \emph{second-countable} if its topology admits a countable basis. The equivalence of the  Type-I and Type-II vague topology on $\mathcal{M}(X)$ with an LCH space $X$ follows essentially from the Urysohn's Lemma.

\bigskip

Proof of Theorem \ref{thm4}:\\

\begin{proof}
It is enough for us to show that the Type-I vague topology is finer than the Type-II vague topology
in case of $X$ being an LCH space. For a measure $\nu\in\mathcal{M}(X)$, consider its neighbourhood $W_{v2}(\nu,f, \epsilon)$ for some continuous function $f: X\rightarrow\mathbb{R}$ vanishing at infinity and some $\epsilon>0$ under the Type-II vague topology. Consider the following two sets in $X$,
\begin{center}
$K=f^{-1}\big((-\infty,-\cfrac{\epsilon}{4}]\cup[\cfrac{\epsilon}{4},\infty)\big)$
\end{center}  
and
\begin{center}
$U=f^{-1}\big((-\infty,-\cfrac{\epsilon}{8})\cup(\cfrac{\epsilon}{8},\infty)\big)$.
\end{center} 
Obviously $K\subset U$ while $K$ is compact and $U$ is open. So according to the Urysohn's Lemma, there exists a continuous map $g_K: X\rightarrow [0,1]$ with compact support $Supp(g_K)\subset U$, such that
\begin{center}
$g_K(x)=1$
\end{center} 
for any $x\in K$. Let $g=f\cdot g_K$. It is a continuous map on $X$ with compact support. We claim that 
the neighbourhood 
\begin{equation}\label{eq10}
W_{v1}(\nu,g, \cfrac{\epsilon}{8})\subset W_{v2}(\nu,f, \epsilon).
\end{equation}
To see this, for any probability measure $\varrho\in W_{v1}(\nu,g, \cfrac{\epsilon}{8})$, we have
\begin{center}
$
\begin{array}{ll}
& |\int_X f(x) d\varrho-\int_X f(x) d\nu|\\ 
= & |\int_X \big(f(x)-g(x)\big) d\varrho+\int_X g(x) d\varrho-\int_X g(x) d\nu-\int_X \big(f(x)-g(x)\big) d\nu|\\
\leq & |\int_X \big(f(x)-g(x)\big) d\varrho|+|\int_X g(x) d\varrho-\int_X g(x) d\nu|+|\int_X \big(f(x)-g(x)\big) d\nu|\\
= & |\int_{K'} \big(1-g_K(x)\big)f(x) d\varrho|+|\int_X g(x) d\varrho-\int_X g(x) d\nu|+|\int_{K'} \big(1-g_K(x)\big)f(x) d\nu|\\
< & \cfrac{3\epsilon}{8},
\end{array}
$
\end{center}
in which $K'=X\setminus K$ is the residual set. (\ref{eq10}) guarantees the Type-I vague topology is finer than Type-II vague topology with LCH ambient spaces. 
\end{proof}

It would be an interesting question to ask whether there exists some none-LCH space $X$, such that the  Type-II vague topology $\mathfrak{W}_{v2}$ is strictly finer than the Type-I vague topology $\mathfrak{W}_{v1}$ on $\mathcal{M}(X)$. We can not give an example of a none-LCH space $X$ for which the strict fineness holds between the two types of vague topology, although there do exist examples of none-LCH spaces $X$ on which the Urysohn's Lemma fails.

Since Urysohn's Lemma holds on normal spaces, the following result follows from the proof of Theorem \ref{thm4}.
\begin{corollary}
The  Type-II vague topology $\mathfrak{W}_{v2}$ is equivalent to the Type-I vague topology $\mathfrak{W}_{v1}$ on $\mathcal{M}(X)$ with $X$ being a normal space.
\end{corollary}

Now we show equivalence of the Type-I and Type-II setwise topology on $\mathcal{M}(X)$.
\bigskip

Proof of Theorem \ref{thm7}:\\

\begin{proof}
Obviously  the topology  $\mathfrak{W}_{s2}$ is finer than  $\mathfrak{W}_{s1}$ since any characteristic function $1_A$ of an measurable set $A\in\mathcal{B}$ is a bounded measurable function. In the following we show the inverse is also true. 

For a probability measure $\nu\in\mathcal{M}(X)$, a function $f\in M_b(X)$ and a small $\epsilon>0$, consider the neighbourhood $W_{s2}(\nu,f, \epsilon)$ of $\nu$. We will find an open neighbourhood $U_{\nu}$ of $\nu$ under $\mathfrak{W}_{s1}$, such that $U_{\nu}\subset W_{s2}(\nu,f, \epsilon)$, this is enough to justify the Type-I setwise topology $\mathfrak{W}_{s1}$ is finer than Type-II setwise topology $\mathfrak{W}_{s2}$.   

First, choose an integer $N\in\mathbb{N}$ large enough such that 
\begin{center}
$\cfrac{4\Vert f\Vert_{\infty}}{N}<\epsilon$.
\end{center}
For $1\leq i\leq N-1$, let 
\begin{center}
$A_i=f^{-1}\big([-\Vert f\Vert_{\infty}+\frac{2(i-1)\Vert f\Vert_{\infty}}{N},-\Vert f\Vert_{\infty}+\frac{2i\Vert f\Vert_{\infty}}{N})\big)$.
\end{center}
Let 
\begin{center}
$A_N=f^{-1}\big([\Vert f\Vert_{\infty}-\frac{2\Vert f\Vert_{\infty}}{N},\Vert f\Vert_{\infty}]\big)$. 
\end{center}
Note that $A_i\in\mathcal{B}$ for any $1\leq i\leq N$ since $f$ is measurable, and $\cup_{1\leq i\leq N}A_i$ is a disjoint partition of the ambient space $X$. Now for every $1\leq i\leq N$, consider the open neighbourhood  $W_{s1}(\nu,A_i, \frac{\epsilon}{2N\Vert f\Vert_{\infty}})$ of $\nu$. Let
\begin{center}
$U_{\nu}=\cap_{1\leq i\leq N}W_{s1}(\nu,A_i, \frac{\epsilon}{2N\Vert f\Vert_{\infty}})$
\end{center} 
be the open neighbourhood of $\nu$ under $\mathfrak{W}_{s1}$. We claim that 
\begin{equation}
U_{\nu}\subset W_{s2}(\nu,f, \epsilon).
\end{equation}
To see this, for any probability measure $\varrho\in U_{\nu}$, compare the integration of $f$ with respect to $\nu$ and $\varrho$ over $X$, we have
\begin{center}
$
\begin{array}{ll}
& |\int_X fd\nu-\int_X fd\varrho|\vspace{2mm}\\
\leq & \sum_{1\leq i\leq N}|\int_{A_i} fd\nu-\int_{A_i} fd\varrho| \vspace{2mm}\\
\leq & \max\Big\{\sum_{1\leq i\leq N}\Big|\big((-\Vert f\Vert_{\infty}+\frac{2i\Vert f\Vert_{\infty}}{N})\nu(A_i)-(-\Vert f\Vert_{\infty}+\frac{2(i-1)\Vert f\Vert_{\infty}}{N})\varrho(A_i)\big)\Big|, \\
&\sum_{1\leq i\leq N}\Big|\big((-\Vert f\Vert_{\infty}+\frac{2(i-1)\Vert f\Vert_{\infty}}{N})\nu(A_i)-(-\Vert f\Vert_{\infty}+\frac{2i\Vert f\Vert_{\infty}}{N})\varrho(A_i)\big)\Big|\Big\} \vspace{2mm}\\
\leq & \sum_{1\leq i\leq N}\Big|\Big(\big(-\Vert f\Vert_{\infty}+\frac{2(i-1)\Vert f\Vert_{\infty}}{N}\big)\big(\nu(A_i)-\varrho(A_i)\big)\Big)\Big|+\frac{2\Vert f\Vert_{\infty}}{N}\vspace{2mm}\\
< & \sum_{1\leq i\leq N}\Vert f\Vert_{\infty}\frac{\epsilon}{2N\Vert f\Vert_{\infty}}+\frac{2\Vert f\Vert_{\infty}}{N}\vspace{2mm}\\
\leq & \frac{\epsilon}{2}+\frac{\epsilon}{2}\vspace{2mm}\\
= & \epsilon.
\end{array}
$
\end{center} 

\end{proof}

\begin{rem}
Be careful that the type-I setwise topology is defined by a basis in Definition \ref{def3} while the type-II setwise topology is defined by a subbasis in Definition \ref{def4} on $\mathcal{M}(X)$. 
\end{rem}

Due to Theorem \ref{thm7} we will usually not distinguish the two types of setwise topology in some cases. However, technically, it is more convenient to resort to one of the two types of setwise topology than the other in due course. These also apply to the two types of vague topology on $\mathcal{M}(X)$ with LCH ambient spaces in virtue of Theorem \ref{thm4}.

\section{Separability and metrizability of the probability space $\mathcal{M}(X)$ under the vague topology and setwise topology}

This section is devoted to properties of the probability space $\mathcal{M}(X)$ under the vague topology and setwise topology, especially the separability and metrizability. We aim to prove Theorem  \ref{thm3}, \ref{thm2} and Theorem  \ref{thm8} in this section. Our strategy is to prove Theorem  \ref{thm8} first by establishing some separation and countability properties of the probability space $\mathcal{M}_{s1}(X)$ ($\mathcal{M}_{s2}(X)$). Since the setwise topology is finer than the vague topology, some separation and countability properties of the probability space $\mathcal{M}_{s1}(X)$ ($\mathcal{M}_{s2}(X)$) are inherited naturally by the space $\mathcal{M}_{v1}(X)$ ($\mathcal{M}_{v2}(X)$) in some cases. These properties are then applied to the proof of  Theorem  \ref{thm3} and \ref{thm2}.

To prove  Theorem  \ref{thm8}, we need several preceding results on the separation and countability of the  topological space $\mathcal{M}_{s1}(X)$ ($\mathcal{M}_{s2}(X)$). In virtue of Theorem \ref{thm7}, all the separation and countability properties are shared by the two spaces $\mathcal{M}_{s1}(X)$ and $\mathcal{M}_{s2}(X)$.

\begin{lemma}\label{lem9}
The topological space $\mathcal{M}_{s1}(X)$ ($\mathcal{M}_{s2}(X)$) is Hausdorff.
\end{lemma}
\begin{proof}
Without loss of generality, suppose that $X$ is not endowed with the trivial topology $\{\emptyset, X\}$. Now for two probability measures $\nu, \varrho\in\mathcal{M}_{s1}(X)$, if  $\nu\neq\varrho$, there must exist some $A\in\mathcal{B}$, such that $\nu(A)\neq\varrho(A)$. Without loss of generality suppose
\begin{center}
$\nu(A)>\varrho(A)$.
\end{center} 
Then we have 
\begin{center}
$\nu\in W_{s1}(\nu,A, \cfrac{\nu(A)-\varrho(A)}{4})$ and $\varrho\in W_{s1}(\varrho,A, \cfrac{\nu(A)-\varrho(A)}{4})$
\end{center}
while
\begin{center}
$W_{s1}(\nu,A, \cfrac{\nu(A)-\varrho(A)}{4})\cap W_{s1}(\varrho,A, \cfrac{\nu(A)-\varrho(A)}{4})=\emptyset$.
\end{center}
\end{proof}

\begin{lemma}\label{lem10}
The topological space $\mathcal{M}_{s2}(X)$ ($\mathcal{M}_{s1}(X)$) is regular.
\end{lemma}
\begin{proof}
Let $\nu\in\mathcal{M}_{s2}(X)$ and $\Xi_1\subset \mathcal{M}_{s2}(X)$ be a closed set such that $\nu\notin \Xi_1$. So the residual set $\Xi_1'=\mathcal{M}_{s2}(X)\setminus \Xi_1$ is an open set such that $\nu\in \Xi_1'$. Then there must exist $f\in M_b(X)$ and $\epsilon>0$, such that
\begin{center}
$W_{s2}(\nu,f,\epsilon)\subset \Xi_1'$. 
\end{center}
Since  
\begin{center}
$W_{s2}(\nu,f,\cfrac{\epsilon}{4})\subset W_{s2}(\nu,f,\cfrac{\epsilon}{2})$ and the closure $\overline{W_{s2}(\nu,f,\epsilon/2)}\subset W_{s2}(\nu,f,\epsilon)$,
\end{center}
we have
\begin{center}
$\nu\in W_{s2}(\nu,f,\cfrac{\epsilon}{4})$ while $C\subset \overline{W_{s2}(\nu,f,\epsilon/2)}'$
\end{center}
and
\begin{center}
$ W_{s2}(\nu,f,\cfrac{\epsilon}{4})\cap \overline{W_{s2}(\nu,f,\epsilon/2)}'=\emptyset$.
\end{center}

\end{proof}

\begin{lemma}\label{lem2}
The topological space $\mathcal{M}_{s1}(X)$ ($\mathcal{M}_{s2}(X)$) is second-countable if the $\sigma$-algebra $\mathcal{B}$ has at most countably many elementary events.
\end{lemma}
\begin{proof}
In case of $\#\mathcal{B}$ being finite, let 
\begin{center}
$\{A_1, A_2,\cdots, A_n\}\subset \mathcal{B}$
\end{center}
be the collection of all the elementary events. Then the set of probability measures
\begin{center}
$\{\nu: \nu(A_i)\in\mathbb{Q}, 0\leq \nu(A_i)\leq 1 \mbox{\ for any\ } 1\leq i\leq n \mbox{\ and\ }\sum_{1\leq i\leq n} \nu(A_i)=1\}$
\end{center}
is a dense subset of $\mathcal{M}_{s1}(X)$. Every measure in the set has a countable neighbourhood basis, which can be used to build a countable basis of $\mathcal{M}_{s1}(X)$. 

Now suppose $\mathcal{B}$ has countably many elementary events $\{A_i\}_{i=1}^\infty$. Consider the collection of finite unions of these elementary events,
\begin{center}
$\{B_F=\cup_{i\in F}A_i\}_{F\subset\mathbb{N},\#F<\infty}$.
\end{center}
It is a countable set. We claim that the countable set of measures
\begin{center}
$\prod_1=\cup_{F\subset\mathbb{N},\#F<\infty,j\in(\mathbb{N}\setminus F)}\{\nu: \nu(B_F)\in\mathbb{Q}\cap[0,1], \nu(A_j)=1-\nu(B_F)\}$
\end{center}
is a countable dense subset in $\mathcal{M}_{s1}(X)$ (one has the freedom to adjust mass on the elementary events in $B_F$, but we take only one such measure with respect to individual $B_F$ in $\prod_1$). 

To see this, let $\varrho\in\mathcal{M}_{s1}(X)$ be a probability measure. For any measurable set $A\in\mathcal{B}$ and any $\epsilon>0$, consider the neighbourhood $W_{s1}(\varrho,A,\epsilon)$. Without loss of generality suppose $A\neq X$ in the following. Now if $A=B_F$ for some  $F\subset\mathbb{N}$ and $\#F<\infty$, obviously there exists some $\nu\in\prod_1$, such that $\nu\in W_{s1}(\varrho,A,\epsilon)$. If 
\begin{center}
$A=\cup_{i=1}^\infty A_{n_i}$
\end{center}
for some $\{n_i\}_{i=1}^\infty\subset\mathbb{N}$, then there exists $k\in\mathbb{N}$ large enough, such that
\begin{center}
$0<\varrho(A)-\varrho(\cup_{i=1}^k A_{n_i})<\cfrac{\epsilon}{4}$.
\end{center}
Let $F^*=\{n_i\}_{i=1}^k$, so $B_{F^*}=\cup_{i=1}^k A_{n_i}$. Then there exists $\nu_{F^*}\in\prod_1$, such that
\begin{center}
$|\nu_{F^*}(B_{F^*})-\varrho(B_{F^*})|<\cfrac{\epsilon}{4}$ 
\end{center}
and 
\begin{center}
$\nu_{F^*}(A\setminus B_{F^*})=0$.
\end{center}
So
\begin{center}
$
\begin{array}{ll}
|\nu_{F^*}(A)-\varrho(A)| & =|\nu_{F^*}(B_{F^*})-\varrho(B_{F^*})+\nu_{F^*}(A\setminus B_{F^*})-\varrho(A\setminus B_{F^*})|\\
&\leq |\nu_{F^*}(B_{F^*})-\varrho(B_{F^*})|+|\nu_{F^*}(A\setminus B_{F^*})-\varrho(A\setminus B_{F^*})|\\
&=|\nu_{F^*}(B_{F^*})-\varrho(B_{F^*})|+|\varrho(A)-\varrho(\cup_{i=1}^k A_{n_i})|\\
&<\epsilon/2.
\end{array}
$
\end{center} 
This implies $\nu_{F^*}\in W_{s1}(\varrho,A,\epsilon)$, and thus justifies our claim.

To see the topological space $\mathcal{M}_{s1}(X)$ is second countable, let 
\begin{center}
$\mathcal{B}_{s}=\{W_{s1}(\nu,A_i,\frac{1}{j}): \nu\in\prod_1, 1\leq i,j<\infty\}$
\end{center}
be the countable family of open sets. The proof that it can serve as a subbase of the topological space $\mathcal{M}_{s1}(X)$ is left to the keen readers.
\end{proof}

Now we are in a position to prove Theorem \ref{thm8}.\\

Proof of Theorem \ref{thm8}:

\begin{proof}
If the Borel $\sigma$-algebra $\mathcal{B}$ admits at most countably many elementary events, then there is a  dense subset of $\mathcal{M}_{s1}(X)$ according to the proof of Lemma \ref{lem2}. If the $\sigma$-algebra $\mathcal{B}$ has uncountably many elementary events, consider the collection of all open neighbourhoods of the Dirac measures $\{\delta_A: A \mbox{\ is an elementary event in\ }\mathcal{B} \}$. One can see that
\begin{center}
$\{W_{s1}(\delta_A,A,1/2): A \mbox{\ is an elementary event in\ }\mathcal{B} \}$
\end{center}  
is an uncountably disjoint family of open sets. So $\mathcal{M}_{s1}(X)$ is not separable in this case.

As to the metrization, if $\mathcal{B}$ has at most countably many elementary events, considering Lemma  \ref{lem9}, \ref{lem10} and \ref{lem2}, $\mathcal{M}_{s1}(X)$ is metrizable by the \emph{Urysohn Metrization Theorem} \cite[Theorem 34.1]{Mun}. If the $\sigma$-algebra $\mathcal{B}$ has uncountably many elementary events, it supports a continuous measure, which does not admit a countable neighbourhood basis according to \cite[Proposition 2.2.1(ii)]{GR}, so  $\mathcal{M}_{s1}(X)$ is not metrizable in this case.
\end{proof}

Now we turn to the proof of Theorem \ref{thm3}. Again we first establish some separation properties of the spaces $\mathcal{M}_{v1}(X)$ and $\mathcal{M}_{v2}(X)$.

\begin{lemma}\label{lem6}
The topological space $\mathcal{M}_{v1}(X)$ ($\mathcal{M}_{v2}(X)$) is Hausdorff if $X$ is LCH.
\end{lemma}
\begin{proof}
For two probability measures $\nu, \varrho\in\mathcal{M}_{v1}(X)$, if  $\nu\neq\varrho$, according to the Riesz Representation Theorem (see for example \cite[Theorem 2.22]{Kallen3} or \cite[Theorem 2.14]{Rud}), there exists some $f\in C_{sc}(X)$, such that 
\begin{center}
$\int_X f d\nu\neq \int_X f d\varrho$.
\end{center}
Without loss of generality suppose
\begin{center}
$\int_X f d\nu< \int_X f d\varrho$.
\end{center} 
Now let $\epsilon=\cfrac{\int_X f d\varrho-\int_X f d\nu}{4}$. One can easily check that 
\begin{center}
$W_{v1}(\nu,f, \epsilon)$ and $W_{v1}(\varrho,f, \epsilon)$
\end{center}
are two disjoint open neighbourhoods of $\nu$ and $\varrho$ respectively.
\end{proof}

\begin{rem}\label{rem1}
Lemma \ref{lem6} needs not to be true without the assumption of $X$ being LCH. For example, consider the affine space $\mathbb{A}^n$ endowed with the Zariski topology admitting infinitely many closed sets. Since any continuous map on $\mathbb{A}^n$ is a constant map in this case, then
\begin{center}
$\int_{\mathbb{A}^n} f d\varrho-\int_{\mathbb{A}^n} f d\nu=0$
\end{center} 
for any $f\in C(\mathbb{A}^n)$ and any two probability measures $\nu, \varrho\in\mathcal{M}(\mathbb{A}^n)$. In this case a neighbourhood $W_{v1}(\nu,f, \epsilon)$ of $\nu$ is always the whole space  $\mathcal{M}(\mathbb{A}^n)$ for any $f\in C(\mathbb{A}^n)$ and $\epsilon>0$.
\end{rem}

\begin{lemma}\label{lem5}
The topological space $\mathcal{M}_{v1}(X)$ ($\mathcal{M}_{v2}(X)$) is regular.
\end{lemma}
\begin{proof}
The proof of Lemma \ref{lem10} applies in these cases, with the role of the measurable function $f\in M_b(X)$ substituted by a continuous function $f\in C_{sc}(X)$ (or $f\in C_0(X)$).  
\end{proof}

Note that Lemma \ref{lem5} holds for any topological space $X$ instead of only for LCH spaces, comparing with Lemma \ref{lem6}.

\begin{lemma}\label{lem8}
The probability spaces $\mathcal{M}_{v1}(X)$ and $\mathcal{M}_{v2}(X)$ are second-countable if the $\sigma$-algebra $\mathcal{B}$ admits at most countably many elementary events.
\end{lemma}
\begin{proof}
First,  if the $\sigma$-algebra $\mathcal{B}$ has at most countably many elementary events, one can check that the set 
\begin{center}
$\prod_2=\cup_{F\subset\mathbb{N},\#F<\infty,j\in(\mathbb{N}\setminus F)}\big\{\nu: \nu(B_i)\in\mathbb{Q}\cap[0,1] \mbox{ for any } i\in F\cup\{j\} \mbox{ and } \sum_{i\in F\cup\{j\}}\nu(A_i)=1\big\}$
\end{center}
is still a dense subset in $\mathcal{M}_{v1}(X)$ and $\mathcal{M}_{v2}(X)$, by a similar argument as in the proof of Lemma \ref{lem2}. Note that the setwise topology $\mathfrak{W}_{s1}$ is finer than the Type-I and Type-II vague topology on $\mathcal{M}(X)$, so $\mathcal{M}_{v1}(X)$ and $\mathcal{M}_{v2}(X)$ are both second-countable in this case.
\end{proof}

Now we are in a position to prove Theorem \ref{thm3}.\\
\bigskip

Proof of Theorem \ref{thm3}:

\begin{proof}
The separability  follows from  the fact that $\prod_2$ is a countable dense subset of $\mathcal{M}_{v1}(X)$ and $\mathcal{M}_{v2}(X)$. The metrizability follows from a combination of Lemma \ref{lem6}, \ref{lem5} and \ref{lem8}, in virtue of the Urysohn Metrization Theorem.
\end{proof}

Now we go towards our final goal in this section-the proof of Theorem \ref{thm2}. Together  with Theorem \ref{thm8}, these results provide some incisive distinctions between the vague (weak) topology and the setwise topology on the probability space $\mathcal{M}(X)$ with a compact metric space $X$. A locally compact and $\sigma$-compact metric space $X$ admits a countable dense subset, say 
\begin{center}
$X_d=\{x_1,x_2,\cdots\}$.
\end{center}
Let $\Xi_2\subset\mathcal{M}(X)$ be the collection of all the discrete probability measures supported on finite points in $X_d$. One can easily check that $\Xi_2$ is separable under either of the two types of vague topology on it.

\begin{proposition}[Tao]\label{pro1}
If $X$ is a locally compact and $\sigma$-compact metric space endowed with a metric $\rho$, then $\Xi_2$ is a dense subset of $\mathcal{M}(X)$ under the Type-I or Type-II vague topology.
\end{proposition}
\begin{proof}
It suffices for us to show that $\Xi_2$ is a dense subset of $\mathcal{M}(X)$ under the Type-II vague topology. Consider a probability measure $\nu\in\mathcal{M}(X)$. First, for any $\epsilon>0$, since $X$ is a locally compact and $\sigma$-compact metric space, there exists some compact set $X_c\subset X$, such that 
\begin{center}
$\nu(X\setminus X_c)<\epsilon$.
\end{center}
Without loss of generality we assume $X_c\neq X$. Since $X_d$ is dense in $X$ and the residual set $X_c'$ of the compact (closed) set $X_c$ is open, we choose some point $x_*\in X_d\cap X_c'$. Let $I\subset\mathbb{N}$ be the collection of all indexes such that $x_i\in X_c$ for $i\in I$. For any $f\in C_0(X)$ and any $\epsilon>0$, there exists some $\delta>0$ independent of $\epsilon$, such that 
\begin{equation}\label{eq2}
|f(x)-f(y)|<\epsilon
\end{equation}
for any $\rho(x,y)<\delta$ and $x,y\in X_c$. For any $x\in X$ and $r>0$, let $B(x,r)$ be the open ball in $X$ centred at $x$ with radius $r$. Let $N_1\in \mathbb{N}$ be large enough such that 
\begin{center}
$\frac{2}{N_1}<\delta$. 
\end{center}
Since $\cup_{i\in I}B(x_i, \frac{1}{N_1})$ covers $X_c$ and $X_c$ is compact, there exists a collection of finite indexes 
\begin{center}
$I_{N_2}=\{i_1, i_2, \cdots, i_{N_2}\}\subset I$ 
\end{center}
for some $N_2\in \mathbb{N}$, such that $\cup_{i\in I_{N_2}}B(x_i, \frac{1}{N_1})$ covers $X_c$. Note that for any $i\in I_{N_2}$ we have
\begin{equation}
\max_{x\in B(x_i, \frac{1}{N_1})\cap X_c} f(x)-\min_{x\in B(x_i, \frac{1}{N_1})\cap X_c} f(x)<\epsilon
\end{equation}  
considering (\ref{eq2}).

Now define a discrete measure $\nu_\epsilon\in \Xi_2$ supported on $\cup_{i\in I_{N_2}}x_i\cup\{x_*\}$ as following.

\begin{itemize}
\item $\nu_\epsilon(\{x_{i_1}\})=\nu\big(B(x_{i_1},\frac{1}{N_1})\cap X_c\big)$.

\item $\nu_\epsilon(\{x_{i_2}\})=\nu\Big(\big(B(x_{i_2},\frac{1}{N_1})\cap X_c\big)\setminus B(x_{i_1},\frac{1}{N_1})\Big)$. 

\item $\nu_\epsilon(\{x_{i_3}\})=\nu\Big(\big(B(x_{i_3},\frac{1}{N_1})\cap X_c\big)\setminus \cup_{j=1}^2 B(x_{i_j},\frac{1}{N_1})\Big)$. 

 $\cdots$

\item $\nu_\epsilon(\{x_{i_{N_2}}\})=\nu\Big(\big(B(x_{i_{N_2}},\frac{1}{N_1})\cap X_c\big)\setminus \cup_{j=1}^{N_2-1} B(x_{i_j},\frac{1}{N_1})\Big)$. 

\item $\nu_\epsilon(\{x_*\})=\nu(X\setminus X_c)$.

\end{itemize}

Compare the integration of $f$ with respect to $\nu$ and $\nu_\epsilon$ over $X$, we have
\begin{center}
$
\begin{array}{ll}
& \int_X fd\nu-\int_X fd\nu_\epsilon\\
= & \int_{X_c} fd\nu-\int_{X_c} fd\nu_\epsilon+\int_{X\setminus X_c} fd\nu-\int_{X\setminus X_c} fd\nu_\epsilon \\
= & \int_{B(x_{i_1},\frac{1}{N_1})\cap X_c} fd(\nu-\nu_\epsilon)+\int_{\big(B(x_{i_2},\frac{1}{N_1})\cap X_c\big)\setminus B(x_{i_1},\frac{1}{N_1})} fd(\nu-\nu_\epsilon)+\cdots\\
& +\int_{\big(B(x_{i_{N_2}},\frac{1}{N_1})\cap X_c\big)\setminus \cup_{j=1}^{N_2-1} B(x_{i_j},\frac{1}{N_1})} fd(\nu-\nu_\epsilon)+\int_{X\setminus X_c} fd(\nu-\nu_\epsilon) \\
\leq & \epsilon\nu_\epsilon(\{x_{i_1}\})+\epsilon\nu_\epsilon(\{x_{i_2}\})+\cdots+\epsilon\nu_\epsilon(\{x_{i_{N_2}}\})+2\epsilon\Vert f\Vert_{\infty}\\
\leq & (1+2\Vert f\Vert_{\infty})\epsilon.
\end{array}
$
\end{center} 

This means that
\begin{center}
$\nu_\epsilon\in W_{v2}\big(\nu,f, (1+2\Vert f\Vert_{\infty})\epsilon\big)$,
\end{center} 
which implies $\Xi_2$ is a dense subset of $\mathcal{M}(X)$ under the Type-II vague topology.

\end{proof}

Equipped with Proposition \ref{pro1}, we are ready to prove Theorem \ref{thm2}.

\bigskip

Proof of Theorem \ref{thm2}:

\begin{proof}
First note that if $X$ is a compact metric space, then according to Theorem \ref{thm4}, it suffices for us to prove the separability and metrizability of either space $\mathcal{M}_{v1}(X)$ or $\mathcal{M}_{v2}(X)$.  If $X$ is a compact metric space, all the probability measures in $\mathcal{M}(X)$ are  Radon measures. 
A compact metric space is of course locally compact  and  $\sigma$-compact, so the conclusion of separability of $\mathcal{M}_{v1}(X)$ follows directly from Proposition \ref{pro1}. 

In the following we show $\mathcal{M}_{v1}(X)$ is metrizable in case $X$ is a compact metric space. Considering Lemma \ref{lem6} and Lemma  \ref{lem5}, in virtue of the Urysohn Metrization Theorem, it suffices for us to show $\mathcal{M}_{v1}(X)$ is second-countable. Let $\Pi_3\subset \Xi_2$ be a countable dense subset of $\mathcal{M}_{v1}(X)$. Since
\begin{center}
$C(X)=C_{sc}(X)=C_0(X)$,
\end{center}
on any compact metric space $X$, according to \cite[Proposition 1.10.20.]{Tao}, $C_{sc}(X)$ is separable. Let $\Xi_3\subset C_{sc}(X)$ be a countable dense subset (note that \cite[Proposition 1.10.20.]{Tao} actually asserts that $C_{sc}(X)$ is separable under the infinite norm on it, this obviously implies the separability under the $L^1$ norm on it). One can easily check that the following collection of open sets is a countable basis under $\mathfrak{W}_{v1}$,
\begin{center}
$\cup_{\nu\in\Pi_3, f\in \Xi_3, n\in\mathbb{N}} W_{v1}(\nu,f, \frac{1}{n})$.
\end{center}
This justifies that $\mathcal{M}_{v1}(X)$ is second-countable.

\end{proof}

Theorem \ref{thm2} together with Theorem \ref{thm8} have some interesting applications to the probability spaces on some popular ambient compact metric spaces.

\begin{corollary}
Let $X=[0,1]$ or $X=\overline{B(\bold{0},1)}\subset\mathbb{R}^n$ be endowed with the Euclidean metric on it. Then the probability space $\mathcal{M}(X)$ is separable and  metrizable under the Type-I or Type-II vague topology, while it is not separable and not metrizable under the Type-I or Type-II setwise topology. 
\end{corollary}

\section{Relative compactness of families of probabilities in $\mathcal{M}(X)$}

In this section we deal with the relative compactness of families of probability measures in  $\mathcal{M}(X)$, especially under the setwise topology.  It seems to us that a general condition for families of probability measures to be relatively compact is difficult when we merely assume the ambient space $X$ is a topological space, so we decide to limit our attention to metric ambient spaces in this section. 

We first recall some results on relative compactness of families of probability measures in  $\mathcal{M}(X)$  under the vague and weak topology. In a similar way as we define setwisely relative compactness, we can define vague, weak or TV relative compactness of families of probability measures in  $\mathcal{M}(X)$. A family $\Xi$  of probability measures in  $\mathcal{M}(X)$ is said to be \emph{tight} if for any small $\epsilon>0$, there exists some compact set $K\subset X$ such that
\begin{center}
$\nu(K)>1-\epsilon$
\end{center}
for any $\nu\in\Xi$. In case of $X$ being a metric space, Prohorov gave the following condition for weakly relative compactness of families of probability measures in  $\mathcal{M}(X)$, see \cite[Theorem 5.1, Theorem 5.2]{Bil1}.

\begin{Prohorov's Theorem}
Let $X$ be a metric space. If a family of probability measures $\Xi\subset \mathcal{M}(X)$ is tight, then it is weakly relatively compact. Conversely, if $X$ is a separable and complete metric space, then $\Xi$ is tight if it is weakly relatively compact.
\end{Prohorov's Theorem}

For the vaguely relative compactness of families of locally finite measures, see \cite[Theorem 4.2]{Kallen1}.

As to the setwisely relative compactness of a family $\Xi\subset\mathcal{M}(X)$, the condition of tightness  is obviously inadequate, even if we require $X$ to be of the best topological space in our consideration.
\begin{exm}\label{exa2}
Let $X=[0,1]$ endowed with the Euclidean metric on it. Let 
\begin{center}
$\Xi_4=\{\delta_{\frac{1}{n}}\}_{n=1}^\infty$
\end{center}
be the sequence of Dirac measures supported on $\{\frac{1}{n}\}$ for an individual $n\in\mathbb{N}$.
\end{exm}

In the above example the ambient space $X$ is a compact, separable, complete metric space, so $\Xi_4$ is tight, while one can not find any setwisely convergent subsequence in $\Xi_4$. The reason of obstacle for the appearance of a setwisely convergent subsequence in $\Xi_4$ is that there are fractures of mass transportation between some open sets and their closed (compact) subsets when taking limit(sup) along the sequence, that is, for the open set $U=(0,1)\subset X$ in Example \ref{exa2}, we have 
\begin{center}
$\limsup_{i\rightarrow\infty} \delta_{\frac{1}{n_i}}(U)=1>\sup_{K\subset U, K \mbox{ is closed}}\limsup_{i\rightarrow\infty} \delta_{\frac{1}{n_i}}(K)=0$
\end{center}
for any subsequence $\{n_i\}_{i=1}^\infty\subset\mathbb{N}$. This inspires us the condition  (\ref{eq3}) on setwisely relative compactness of families of probability measures in  $\mathcal{M}(X)$.

\begin{lemma}\label{lem4}
A compact metric space $X$ admits a countable base  $\mathcal{B}_c$ such that if $x\in U$ for some open set $U\subset X$, then there is some $B\in\mathcal{B}_c$ such that 
\begin{center}
$x\in B\subset \bar{B}\subset U$,
\end{center}
in which $\bar{B}$ is the closure of $B$.
\end{lemma}
\begin{proof}
First, since a compact metric space is second-countable, we can find a countable base $\mathcal{B}_{c1}$.  Then according to \cite[p237, Theorem]{Bil1}, there is a countable collection $\mathcal{B}_{c2}$ of open sets such that  if $x\in U$ for some open set $U\subset X$, then there is some $B\in\mathcal{B}_{c2}$ such that 
\begin{center}
$x\in B\subset \bar{B}\subset U$.
\end{center}
Then the countable base $\mathcal{B}_c=\mathcal{B}_{c1}\cup\mathcal{B}_{c2}$ satisfies the requirement of the the lemma.
\end{proof}

\bigskip

Proof of Theorem \ref{thm1}:

\begin{proof}
First, if a family of probability measures $\Xi$ is setwisely relatively compact,  then for any sequence $\{\nu_n\}_{n=1}^\infty\subset \Xi$, we can find a subsequence $\{\nu_{n_i}\}_{i=1}^\infty$, such that
\begin{center}
$\nu_{n_i}\stackrel{s}{\rightarrow}\nu$ 
\end{center}
as $i\rightarrow\infty$ for some probability measure $\nu\in\mathcal{M}(X)$. Since 
\begin{center}
$\lim_{i\rightarrow\infty} \nu_{n_i}(K)=\nu(K)$ and $\lim_{i\rightarrow\infty} \nu_{n_i}(U)=\nu(U)$
\end{center}
for any open set $U\subset X$ and closed (any closed set is compact since we are assuming the ambient space $X$ is compact now) set $K\subset U$, we have
\begin{center}
$\lim_{i\rightarrow\infty} \nu_{n_i}(U)=\sup_{K\subset U, K \mbox{ is closed}}\lim_{i\rightarrow\infty} \nu_{n_i}(K)$
\end{center}
as $\nu$ is a regular measure on the metric space $X$. This of course guarantees that the subsequence $\{\nu_{n_i}\}_{i=1}^\infty$ satisfies (\ref{eq3}). 

Now we show the inverse is also true. Suppose for any sequence of probability measures $\{\nu_n\}_{n=1}^\infty\subset \Xi$, we can find a subsequence $\{\nu_{n_i}\}_{i=1}^\infty$ satisfying (\ref{eq3})
for any open set $U\subset X$. Upon the technique in constructing a weakly convergent subsequence of probability measures in proving the Prohorov's Theorem \cite[p60]{Bil1}, we will find a setwisely convergent subsequence of the sequence $\{\nu_{n_i}\}_{i=1}^\infty$.  Since $X$ is a compact metric space, it is second-countable. In virtue of Lemma \ref{lem4}, let $\mathcal{B}_c=\{B_i\}_{i=1}^\infty$ be a countable basis such that if $x\in U$ for some open set $U\subset X$, there is some $j\in\mathbb{N}$ such that 
\begin{center}
$x\in B_j\subset \bar{B}_j\subset U$,
\end{center}
in which $\bar{B}_j$ is the closure of $B_j$. Let 
\begin{center}
$\mathcal{\bar{B}}_c=\big\{B: B=\cup_{j=1}^n \bar{B}_{i_j} \mbox{ with } \{i_j\}_{j=1}^n\subset\mathbb{N}\big\}$.
\end{center}
Now consider the sequence  $\{\nu_{n_i}\}_{i=1}^\infty$ satisfying (\ref{eq3}). Since $\mathcal{\bar{B}}_c$ is countable, we can find a subsequence $\{\nu_{n_{i_j}}\}_{j=1}^\infty$ of $\{\nu_{n_i}\}_{i=1}^\infty$, such that $\lim_{j\rightarrow\infty} \nu_{n_{i_j}}(B)$ exists for any $B\in\mathcal{\bar{B}}_c$. Then we can find a probability measure $\nu\in\mathcal{M}(X)$, such that
\begin{equation}\label{eq4}
\nu(U)=\sup_{B\subset U, B\in\mathcal{\bar{B}}_c}\lim_{j\rightarrow\infty} \nu_{n_{i_j}}(B)\leq\liminf_{j\rightarrow\infty} \nu_{n_{i_j}}(U)
\end{equation}
for any open set $U\subset X$. We claim now that the limit of the sequence $\lim_{j\rightarrow\infty} \nu_{n_{i_j}}(U)$ exists for any open set $U\subset X$, moreover, we have
\begin{equation}\label{eq7}
\nu(U)=\lim_{j\rightarrow\infty} \nu_{n_{i_j}}(U)
\end{equation}  
for any open set $U\subset X$. This is enough to justify 
\begin{center}
$\nu_{n_{i_j}}\stackrel{s}{\rightarrow}\nu$ 
\end{center}
as $j\rightarrow\infty$ in virtue of Theorem \ref{thm9}, which completes the proof. To show the claim, note that the sequence $\{\nu_{n_{i_j}}\}_{j=1}^\infty$ satisfies
\begin{equation}\label{eq5}
\limsup_{j\rightarrow\infty} \nu_{n_{i_j}}(U)=\sup_{K\subset U, K \mbox{ is closed}}\limsup_{j\rightarrow\infty} \nu_{n_{i_j}}(K)
\end{equation}
for any open set $U\subset X$. Since any metric space is normal \cite{Tao}, for any closed set $K\subset U$,  we can find an open set $U_K$, such that $K\subset U_K\subset \bar{U}_K\subset U$. As $U_K$ can be written as an union of sets in $\mathcal{B}_c$ and $K$ is compact, then $U_K$ can be written as a finite union of sets in $\mathcal{B}_c$ whose closures are all in $U$. This means that there exists some $B\in \bar{B}_c$ such that $K\subset B$, which guarantees that  
\begin{equation}\label{eq6}
\sup_{B\subset U, B\in\mathcal{\bar{B}}_c}\lim_{j\rightarrow\infty} \nu_{n_{i_j}}(B)=\sup_{K\subset U, K \mbox{ is closed}}\limsup_{j\rightarrow\infty} \nu_{n_{i_j}}(K)
\end{equation}
for any open set $U\subset X$. Now combining (\ref{eq4}) (\ref{eq5}) and (\ref{eq6}) together, we have
\begin{center}
$\limsup_{j\rightarrow\infty} \nu_{n_{i_j}}(U)=\nu(U)\leq\liminf_{j\rightarrow\infty} \nu_{n_{i_j}}(U)$
\end{center} 
for any open set $U\subset X$, which implies our claim and (\ref{eq7}).

\end{proof}

It seems that checking the condition (\ref{eq3}) holding for any open $U\subset X$ is a rather tough job in Theorem \ref{thm1}, in fact, we only need to check it holds for countably many open  $U\subset X$. 

\begin{corollary}\label{cor2}
For a compact metric space $X$, there exists a countable collection $\mathcal{B}_{cf}$ of open sets in $X$, such that a family of probability measures $\Xi\subset \mathcal{M}(X)$ is setwisely relatively compact if and only if for any sequence of probability measures $\{\nu_n\}_{n=1}^\infty\subset \Xi$, there is a subsequence $\{\nu_{n_i}\}_{i=1}^\infty$, such that (\ref{eq3}) holds for any $U\in\mathcal{B}_{cf}$.
\end{corollary}

\begin{proof}
It suffices for us to show the sufficiency. In case of $X$ being a compact metric space, let $\mathcal{B}_{c1}$ be a countable base. Now let
\begin{center}
$\mathcal{B}_{cf}=\{B: B \mbox{ is a finite union of sets in }\mathcal{B}_{c1}\}$.
\end{center}
$\mathcal{B}_{cf}$ is a countable set. Now suppose that for any sequence of probability measures $\{\nu_n\}_{n=1}^\infty\subset \Xi$, (\ref{eq3}) holds for some subsequence $\{\nu_{n_i}\}_{i=1}^\infty$ on any open set in $\mathcal{B}_{cf}$. Since $\mathcal{B}_{cf}$ is countable, we can find a subsequence $\{\nu_{n_{i_j}}\}_{j=1}^\infty$ of $\{\nu_{n_i}\}_{i=1}^\infty$, such that $\lim_{j\rightarrow\infty} \nu_{n_{i_j}}(B)$ exists for any $B\in\mathcal{B}_{cf}$.   For any open set $U\subset X$, let 
\begin{center}
$U=\cup_{j=1}^\infty A_j$
\end{center} 
with $A_j\in\mathcal{B}_{c1}$ for any $j\in\mathbb{N}$. One can show that
\begin{equation}\label{eq8}
\limsup_{i\rightarrow\infty} \nu_{n_i}(U)=\lim_{m\rightarrow\infty}\limsup_{i\rightarrow\infty} \nu_{n_i}(\cup_{j=1}^m A_j).
\end{equation}
Note that $\cup_{j=1}^m A_j\in\mathcal{B}_{cf}$. According to the assumption, we have
\begin{equation}\label{eq9}
\begin{array}{ll}
&\lim_{m\rightarrow\infty}\limsup_{i\rightarrow\infty} \nu_{n_i}(\cup_{j=1}^m A_j)\\
=&\lim_{m\rightarrow\infty}\sup_{K\subset \cup_{j=1}^m A_j, K \mbox{ is closed}}\limsup_{i\rightarrow\infty} \nu_{n_i}(K)\\
\leq & \sup_{K\subset U, K \mbox{ is closed}}\limsup_{i\rightarrow\infty} \nu_{n_i}(K).
\end{array}
\end{equation}
Now combining (\ref{eq8}), (\ref{eq9}) together with the following obvious fact
\begin{equation}
\limsup_{i\rightarrow\infty} \nu_{n_i}(U)\geq \sup_{K\subset U, K \mbox{ is closed}}\limsup_{i\rightarrow\infty} \nu_{n_i}(K),
\end{equation} 
we justify (\ref{eq3}) holds for the subsequence $\{\nu_{n_i}\}_{i=1}^\infty$ and any open set $U\subset X$. This finishes the proof in virtue of Theorem \ref{thm1}.
\end{proof}

One is recommended to compare Corollary \ref{cor2} with \cite[Theorem 3.3.]{FKZ1} and \cite[Lemma 3.4.]{FKZ1}. 

Theorem \ref{thm1} and Corollary \ref{cor2} of course can be applied to judge setwise convergence of sequences of probability measures in due course, with the aid of the following result. 

\begin{lemma}\label{lem3}
For any topological space $X$, a sequence of probability measures $\{\nu_n\}_{n=1}^\infty\subset\mathcal{M}(X)$ converges vaguely, setwisely or TV to some $\nu\in\mathcal{M}(X)$ as $n\rightarrow\infty$ if and only if every subsequence $\{\nu_{n_i}\}_{i=1}^\infty$ of $\{\nu_n\}_{n=1}^\infty$ contains a further subsequence  $\{\nu_{n_{i_j}}\}_{j=1}^\infty$, such that $\{\nu_{n_{i_j}}\}_{j=1}^\infty$ converges vaguely, setwisely or TV to the probability measure $\nu$ as $j\rightarrow\infty$ respectively.
\end{lemma}

\begin{proof}
The necessity is obvious, while the sufficiency of the result follows from a simple argument of reduction to absurdity, see \cite[Theorem 2.6.]{Bil1}.
\end{proof}

Theorem \ref{thm1} together with Lemma \ref{lem3} indicate the following result on judging the setwise convergence of sequences of probability measures in $\mathcal{M}(X)$ with the ambient space being a compact metric space.

\begin{corollary}
Let $X$ be a compact metric space. Now if for a sequence of probability measures $\{\nu_n\}_{n=1}^\infty\subset\mathcal{M}(X)$, there is a subsequence $\{\nu_{n_i}\}_{i=1}^\infty$, such that
\begin{equation}\label{eq3}
\limsup_{i\rightarrow\infty} \nu_{n_i}(U)=\sup_{K\subset U, K \mbox{ is closed}}\limsup_{i\rightarrow\infty} \nu_{n_i}(K)
\end{equation}
for any open set $U\subset X$, and every setwisely convergent subsequence converges to the same probability measures $\nu\in\mathcal{M}(X)$, then  
\begin{center}
$\nu_n\stackrel{s}{\rightarrow}\nu$ 
\end{center}
as $n\rightarrow\infty$.
\end{corollary}

\section{Some further discussions on the topology on $\mathcal{M}(X)$}

In this section we intend to make a discussion on more kinds of topology on $\mathcal{M}(X)$, their relationships and their topological properties. We will formulate some open problems on our concerns.

One problem is on the comparison between various other topology on $\mathcal{M}(X)$. In the above sections we make a comparison between the the two types of vague topology, as well as the two types of setwise topology on the probability space $\mathcal{M}(X)$. Similar questions arise on the comparison between other kinds of topology on  $\mathcal{M}(X)$. For example, considering the Portemanteau theorem (see for example \cite[Theorem 13.16]{Kle}, \cite[Theorem 2.1]{Bil1} and \cite[Theorem 1.4.16]{HL1}), we define the following three types of weak topology on $\mathcal{M}(X)$ with the ambient space $X$ being a general topological space.

\begin{definition}
The \emph{Type-II weak topology} $\mathfrak{W}_{w2}$ on $\mathcal{M}(X)$ is the topology with subbasis
\begin{center}
$W_{w2}(\nu, A, \epsilon)=\{\varrho\in\mathcal{M}(X): |\varrho(A)-\nu(A)|<\epsilon \}$
\end{center}
for $A\in\mathcal{B}$ with $\nu(\partial A)=0$ and any real $\epsilon>0$. 
\end{definition}

\begin{definition}
The \emph{Alexandrov topology} $\mathfrak{W}_{w3}$ on $\mathcal{M}(X)$ is the topology with subbasis 
\begin{center}
$W_{w3}(\nu, A, \epsilon)=\{\varrho\in\mathcal{M}(X): \varrho(A)>\nu(A)-\epsilon \}$
\end{center}
for any open set $A\in\mathcal{B}$ and any real $\epsilon>0$. 
\end{definition}

Refer to \cite{Ale}, \cite{Bla} and \cite{Kal} for the Alexandrov topology on $\mathcal{M}(X)$. 

\begin{definition}
The \emph{Type-IV weak topology} $\mathfrak{W}_{w4}$ on $\mathcal{M}(X)$ is the topology with subbasis 
\begin{center}
$W_{w3}(\nu, A, \epsilon)=\{\varrho\in\mathcal{M}(X): \varrho(A)<\nu(A)+\epsilon \}$
\end{center}
for any closed set $A\in\mathcal{B}$ and any real $\epsilon>0$.  
\end{definition}

These types of weak topology are induced from  equivalent descriptions of the \emph{weakly sequential convergence} of measures in $\mathcal{M}(X)$ with a metric ambient space $X$, refer to \cite{Bil1, Bil2, Kallen2, Kallen5, Kle, Las2, Mat}.

\begin{pro}\label{prob1}
How is the comparison of the fineness between the Type-I, Type-II, the Alexandrov and Type-IV weak topology on $\mathcal{M}(X)$ with $X$ being a topological space?
\end{pro}

Another concern is on the properties of $\mathcal{M}(X)$ under various topology, for example, its separation properties, metrizability or completeness. Our Theorem \ref{thm3}, Theorem \ref{thm2} together with \cite[Theorem 4.2]{Kallen1} provide some conditions on the separability and metrizability of the vague topology on $\mathcal{M}(X)$ with certain assumptions on the ambient space. A remaining problem is the following one.

\begin{pro}\label{prob2}
In case that the ambient space $X$ is not a separable and complete metric space but its Borel $\sigma$-algebra $\mathcal{B}$ admits uncountably many elementary events, is the probability space $\mathcal{M}_{v1}(X)$ or $\mathcal{M}_{v2}(X)$ separable or metrizable?  
\end{pro}

It is well-known that the Type-II weak topology on the probability space $\mathcal{M}(X)$ can be induced by the \emph{Prohorov metric} if the ambient space $X$ is a separable and complete metric space, see \cite[p72]{Bil1}. According to Theorem \ref{thm8}, the probability space $\mathcal{M}(X)$ is metrizable in case the Borel $\sigma$-algebra $\mathcal{B}$ admits at most countably many elementary events, then it is natural to try to give an explicit metric which induces the setwise topology on  $\mathcal{M}(X)$ in  this cases.

\begin{pro}
In the context of Theorem \ref{thm8}, can one give an explicit metric which induces the setwise topology in case the probability space $\mathcal{M}(X)$ is metrizable? How is its relationship with the Prohorov metric and the TV metric on $\mathcal{M}(X)$?   
\end{pro}

We have got a necessary and sufficient condition on setwisely relative compactness of families of probabilities in $\mathcal{M}(X)$ with the ambient space $X$ being a compact metric space. The question on non-compact metric space or even non-metrizable space is still open.

\begin{pro}\label{prob3}
Can one give a necessary and sufficient condition for a family of probability measures in $\mathcal{M}(X)$ to be setwisely relatively compact in case the ambient space $X$ is a non-compact metric space or even a non-metrizable space?  
\end{pro}

We suspect that on a non-compact metric space, tightness of the family together with existence of a subsequence satisfying (\ref{eq3}) on any open $U\subset X$ for any sequence in the family may be sufficient for setwisely relative compactness of the family? 


\section{General $F$-topology and $S$-topology on $\mathcal{M}(X)$}

We study topological properties of $\mathcal{M}(X)$ under the vague, weak, setwise or TV topology in the above sections, however, it is possible that these four kinds of topologies are still not sensitive enough to deal with ones' concerning problems in due courses.  The concepts of $F$-topology and $S$-topology on $\mathcal{M}(X)$ allows one to work under its ideal topology on $\mathcal{M}(X)$ in one's need.

\begin{definition}
Let $F$ be some family of functions from $X$ to $\mathbb{R}$.  The $F$-topology $\mathfrak{W}_{F}$ on $\mathcal{M}(X)$ is the topology with basis 
\begin{center}
$W_{F}(\nu,f, \varepsilon)=\{\varrho\in\mathcal{M}(X): |\int_X f(x) d\varrho-\int_X f(x) d\nu|<\epsilon \}$
\end{center}
for any $f\in F$ and any real $\epsilon>0$. 
\end{definition}

\begin{definition}
Let $S\subset\mathcal{B}$ be some family of Borel sets. The $S$-topology $\mathfrak{W}_{S}$ on $\mathcal{M}(X)$ is the topology with subbasis
\begin{center}
$W_{S}(\nu, A, \epsilon)=\{\varrho\in\mathcal{M}(X): |\varrho(A)-\nu(A)|<\epsilon \}$
\end{center}
for any $A\in S$ and any real $\epsilon>0$. 
\end{definition}

By suitable choices of different families of functions $F$ or sets $S$, one can define the induced topology on $\mathcal{M}(X)$, which may crucially result in solutions of one's problems in due courses. Theoretically, it is also interesting to ask how to interpret the $F$-topology in words of the $S$-topology, or vice versa, for particular families of $F$ or $S$. Some algebraic structures (such as lattice structure) are possible on the collections of the $F$-topology and the $S$-topology, in view of the natural algebraic structures on the collections of families $F$ and $S$ separately. Separability, metrizability and relative compactness will vary depending on different choices of families of functions $F$ or sets $S$.

\end{document}